\documentclass[12pt,a4paper, reqno]{amsart}
\usepackage{amsfonts,amsmath,amsthm,amssymb, amscd}
\usepackage[top=2cm, bottom=5cm, left=3cm, right=1cm]{geometry}
\def\gp#1{\langle#1\rangle}

\usepackage[active]{srcltx}

\def\gp#1{\langle #1 \rangle}

\newtheorem*{theorem}{Theorem}
\newtheorem*{corollary}{Corollary}

\title[Locally nilpotent unitary unit group]{Modular group algebras whose group of  unitary units \\ is locally nilpotent}

\author{V.~Bovdi}

\thanks{
Supported by the UAEU  UPAR  grant:  G00002160}

\address{Department of Math. Sciences\\
UAE University\\
Al-Ain\\
United Arab Emirates}
\email{vbovdi@gmail.com}

\begin{document}
\maketitle

\begin{abstract}
We characterize those modular group algebras $FG$ whose  group of unitary units is locally nilpotent under the classical involution  of $FG$.
\end{abstract}

Let $V_*(KG)$ be the unitary subgroup of the group $V(KG)$  of normalized units of the  group ring $KG$ of a group $G$ over the ring $K$,  under the classical involution $*$ of $KG$. The group  $V_*(KG)$  has a complicated  structure, has been actively studied  and it has  several applications (for instance, see \cite{involut_book, Novikov,  Serre}). For an overview  we  recommend the survey paper \cite{Bovdi_survey}.

\smallskip

Our main result is the following.

\begin{theorem}\label{T:1}
Let $V(FG)$ be the group of normalized units of the  modular group algebra $FG$ of a group $G$ over the field $F$  of positive characteristic $p$. Let $V_*$ be the unitary subgroup of the group $V(FG)$ under the classical involution $*$ of $FG$. The  following conditions are equivalent:
\begin{itemize}
\item[(i)] $V(FG)$   is locally nilpotent;
\item[(ii)] $V_*(FG)$   is locally nilpotent;
\item[(iii)] $G$ is locally nilpotent and  the commutator subgroup $G'$ of $G$ is a $p$-group.
\end{itemize}
\end{theorem}

\begin{proof} (ii) $\Rightarrow$ (iii).
Since  $V_*$ is  a locally nilpotent group,   $G$  is also locally nilpotent.
Consider an f.g.  subgroup $S=\gp{f_1,\ldots, f_s\in V_*}\leq V_*$.
Clearly $H=\gp{supp(f_1),\ldots, supp(f_s)}$
is an f.g. nilpotent  subgroup of $G$ and $S\leq V_*(FH)< V_*(FG)$.
Hence  we may restrict our attention to the subgroup $V_*(FH)$, where  $H$ is an f.g. nilpotent subgroup  of $G$ containing a $p$-element. The set  $\mathfrak{t}(H)$ of torsion elements of $H$ is a finite group (see   \cite{Hall_2}, 7.7, p. 29)  and  $\mathfrak{t}(H)=\times_{q} \mathrm{S}_q$ is a direct product  of its Sylow  $q$-subgroups  $\mathrm{S}_q$ (see  \cite{Hall_book}, 10.3.4, p.176).   Clearly   there exists  $c\in \zeta(\mathrm{S}_p)$ of order $p$ such that  $\widehat{c}=\sum_{i=0}^{p-1}c^i$   is a central square-zero  element of $FG$, where $\zeta(\mathrm{S}_p)$ is the center of $\mathrm{S}_p$. Let us  fix  $c\in \zeta(\mathrm{S}_p)$.

Now we prove that for any $g,h\in G$ with  $(g,h)\not=1$ there always exists $s\in \mathbb{N}$ such that $h^{p^s}\in  C_G(g)$.  We consider  the following three cases:

\smallskip

\noindent
\underline{Case 1. Let $g^2\not\in\gp{c\mid c^p=1}\subseteq \zeta(\mathrm{S}_p)$.} Since  a locally nilpotent group  is always   Engel (see \cite{Traustason}) and the element  $w=1+(g-g^{-1})\widehat{c}$ is a unitary unit,   $L=\gp{g,h, w}$ is an f.g. nilpotent subgroup of $V_*$. Hence  $L$ is Engel and    the nilpotency class $cl(L)$   of $L$ is at most $p^m$ for some $m\in \mathbb{N}$. Put   $q=p^m$.  Since $\binom {q}{i} \equiv 0 \pmod{p}$ for  $0 <i< q$ and $L$ is   Engel, we get that
\[
\begin{split}
1=\big(w,h, q\big)&= 1+\widehat{c}\sum_{i=0}^q (-1)^i\textstyle\binom {q}{i} \Big(g^{h^{q-i}}-g^{-h^{q-i}}\Big)\\
&=1+\widehat{c}\Big((g^{h^{q}}-g)-(g^{-h^{q}}-g^{-1})\Big),
\end{split}
\]
so
\[
\begin{split}
\widehat{c}\big((g^{h^{q}}-g)-(g^{-h^{q}}&-g^{-1})\big)=\\
&=\widehat{c}\big(g^2(g,{h^{q}})-\underline{g^2}-(g^{-1},g^{-h^{q}})+1\big)=0,
\end{split}
\]
which leads to the following  two cases:\quad either   $g^2=g^2(g,{h^{q}}) c^j$\; or\;  $g^2=(g^{-1},g^{-h^{q}}) c^j$.

\smallskip

If  $g^2=g^2(g,{h^{q}}) c^j$ for some  $0\leq j<p$, then  $(g,h^{q})=c^i$  and  $g^{h^{q}}=c^ig$ for some  $0\leq i<p$. Hence $\big(h^{-q}\big)^p g\big(h^{q}\big)^p=(c^i)^pg=g$ and $h^{p^{m+1}}\in C_G(g)$. If   $g^2=(g^{-1},g^{-h^{q}}) c^j$  for some  $0\leq j<p$, then similarly to the previous case  we get  that $h^{p^{m+1}}\in  C_G(g)$.


\smallskip

\noindent
\underline{Case 2. Let  $g^2\in \gp{c\mid c^2=1}\subseteq \zeta(\mathrm{S}_2)$ and $p=2$.} Clearly   $|g|\in\{2,4\}$  and  $w=1+g\widehat{c}\in V_*(FG)$, so $M=\gp{g,h, w}$ is an Engel subgroup  of $V_*(FG)$ and
there exists $m\in \mathbb{N}$ such that  $cl(M)\leq 2^m$. Put   $k=2^m$.  Since $\binom {k}{i} \equiv 0 \pmod{2}$ for  $0 <i< k$ and $M$ is   Engel, we get that
\[
1=\big(w,h, k\big)= 1+\widehat{c}\sum_{i=0}^k (-1)^i\textstyle\binom {k}{i} g^{h^{q-i}}=1+\widehat{c}\big(g^{h^{q}}-g\big).
\]
Hence    $\widehat{c}(g^{h^{k}}-g)=0$ and $(g,{h^{k}})\in\gp{c}$, so either  $(g,{h^{k}})=1$ or     $g^{h^{k}}=cg$.
In the second case we have\quad  $\big(h^{-k}\big)^2 g\big(h^{k}\big)^2=c^2g=g$\quad  and \quad   $h^{2^{m+1}}\in C_G(g)$.

\smallskip

\noindent
\underline{Case 3. Let  $g^2\in \gp{c}\subseteq \zeta(\mathrm{S}_p)$ and $p>2$.} Evidently   $|g|=2p$ and $g=ax$ in which  $|a|=2$ and $x\in \gp{c}$. Moreover $g^{p^t}=a$ for any $t\in \mathbb{N}$ and $(g,h)=( g^{p^t},h)=(a,h)\not=1$. Now exchanging $g\leftrightarrows h$ and repeating the calculation above (i.e. the cases when either $h^2\not\in\gp{c}$ or $h^2\in\gp{c}$ and $p>2$) we see that the   only possible case left is  $|h|=2p$ and $h^2\in\gp{c}$.  It follows that $h=by$ in which $|b|=2$ and $y\in \gp{c}$. Thus  $(g,h)=(a,b)\not=1$, $|ab|>2$ and $a(ab)a=(ab)^{-1}$, so
\[
\gp{1+\big((ab)-(ab)^{-1}\big)\widehat{c},\;  a}\cong C_p\rtimes C_{2}, \quad (p>2)
\]
is a non-nilpotent unitary subgroup, a contradiction.

Consequently   for any $g,h\in G$ with  $(g,h)\not=1$ there always exists $s\in \mathbb{N}$ such that
$h^{p^s}\in  C_G(g)$, so   $(g,h)$ is  a  $p$-element  (see \cite{Robinson}, 10.1.4,  p.287) and  $G'$ is a $p$-group.

If  $G$ is locally nilpotent with    $G'$ is a $p$-group, then $V(FG)$ is locally nilpotent (see \cite{Bovdi_Locally_nilpotent}, Theorem 3), so $V_*\leq V(FG)$ is  locally nilpotent too.  This proves  (iii) $\Rightarrow$ (i) $\Rightarrow$ (ii).
\end{proof}

\smallskip

The following  trivial  consequence of our  theorem  gives  a generalization of a  result of G.T.~Lee, S.~Sehgal and E.~Spinelli  obtained in \cite{Lee_Sehgal_Spinelli} for the case when the  characteristic of the field $F$ is different from $2$.

\begin{corollary}\label{C:1}
Let $FG$ be  the  modular group algebra  of a group $G$ over the field $F$  of positive characteristic $p$.  Let $V_*$ be the unitary subgroup of the group $V(FG)$ under the classical involution $*$ of $FG$. The  following conditions are equivalent:
\begin{itemize}
\item[(i)] $V(FG)$   is  nilpotent;
\item[(ii)] $V_*(FG)$   is  nilpotent;
\item[(iii)] $G$ is  nilpotent and  the commutator subgroup $G'$ of $G$ is a finite $p$-group.
\end{itemize}
\end{corollary}

\end{document}